\documentclass{amsart}
\usepackage{amssymb}
\usepackage{amsfonts}

\newtheorem{theorem}{Theorem}
\theoremstyle{plain}

\newtheorem{corollary}{Corollary}

\newtheorem{definition}{Definition}

\newtheorem{lemma}{Lemma}

\newtheorem{proposition}{Proposition}

\numberwithin{equation}{section}
\input{tcilatex}

\begin{document}
\title[Simpson type inequalities]{Simpson type inequalities for $Q-$ Class
functions }
\author{M.Emin \"{O}zdemir$^{\blacklozenge }$}
\address{$^{\blacklozenge }$Atat\"{u}rk University, K.K. Education Faculty,
Department of Mathematics, Erzurum 25240, Turkey}
\email{emos@atauni.edu.tr}
\author{Alper Ekinci$^{\clubsuit }$}
\email{alperekinci@hotmail.com}
\author{Mustafa G\"{u}rb\"{u}z$^{\clubsuit }$}
\address{$^{\spadesuit }$A\u{g}r\i\ \.{I}brahim \c{C}e\c{c}en University,
Faculty of Education, Department of Mathematics, A\u{g}r\i\ 04100, Turkey}
\email{mgurbuz@agri.edu.tr}
\author{Ahmet Ocak Akdemir$^{\clubsuit }$}
\address{$^{\clubsuit }$A\u{g}r\i\ \.{I}brahim \c{C}e\c{c}en University,
Faculty of Science and Arts, Department of Mathematics, A\u{g}r\i\ 04100,
Turkey}
\email{ahmetakdemir@agri.edu.tr}
\keywords{Simpson inequality, $Q-$Class functions, power-mean inequality,
Simpson's formulae.}

\begin{abstract}
In this paper, we obtain some Simpson type inequalities for functions whose
second derivatives' absolute value or q-th power of them are $Q-$class
functions. Also we give applications to numerical integration.
\end{abstract}

\maketitle

\section{introduction}

Suppose $f:[a,b]\rightarrow 
\mathbb{R}
$ is a four times continuously differentiable mapping on $(a,b)$ and $%
\left\Vert f^{(4)}\right\Vert _{\infty }=\sup \left\vert
f^{(4)}(x)\right\vert <\infty .$ The following inequality%
\begin{equation*}
\left\vert \frac{1}{3}\left[ \frac{f(a)+f(b)}{2}+2f\left( \frac{a+b}{2}%
\right) \right] -\frac{1}{b-a}\int_{a}^{b}f(x)dx\right\vert \leq \frac{1}{%
2880}\left\Vert f^{(4)}\right\Vert _{\infty }\left( b-a\right) ^{4}
\end{equation*}%
is well known in the literature as Simpson's inequality.

For some results about Simpson inequality see \cite{ADS}-\cite{U2}.

Recall the Definiton of Godunova-Levin function as following:

\begin{definition}
(See \cite{GO}, \cite[p.410]{MIT3}) We say that $f:I\rightarrow \mathbb{R}$\
is a Godunova-Levin function or that $f$\ belongs to the class $Q\left(
I\right) $ if $f$\ is non-negative and for all $x,y\in I$\ and $t\in \left(
0,1\right) $ we have 
\begin{equation*}
f\left( tx+\left( 1-t\right) y\right) \leq \frac{f\left( x\right) }{t}+\frac{%
f\left( y\right) }{1-t}.
\end{equation*}
\end{definition}

The main aim of this paper is to give some new inequalities of Simpson's
type for functions whose second derivatives of absolute values are
Godunova-Levin functions or belong to the class $Q\left( I\right) .$

\section{Main Results}

We used the following Lemma to obtain our main results.

\begin{lemma}
(See \cite{U2}) Let $f:I\subset 
\mathbb{R}
\rightarrow 
\mathbb{R}
$ be an absolutely continuous mapping on I$^{0}$ where $a,b\in I$ with $a<b,$
such that $f^{\prime \prime }\in L\left[ a,b\right] .$ Then the following
equality holds:%
\begin{eqnarray*}
&&\left\vert \frac{1}{b-a)}\int_{a}^{b}f(x)dx-\frac{1}{6}\left[
f(a)+4f\left( \frac{a+b}{2}\right) +f(b)\right] \right\vert \\
&=&(b-a)^{2}\int_{0}^{1}p(t)f^{\prime \prime }(tb+\left( 1-t\right) a)dt,
\end{eqnarray*}%
where%
\begin{equation*}
\ \ p(t)=\left\{ 
\begin{array}{c}
\frac{1}{6}t\left( 3t-1\right) ,\text{ \ \ \ \ }t\in \left[ 0,\frac{1}{2}%
\right) \\ 
\\ 
\frac{1}{6}\left( t-1\right) \left( 3t-2\right) ,\text{ \ \ \ \ }t\in \left[ 
\frac{1}{2},1\right]%
\end{array}%
\right. .
\end{equation*}
\end{lemma}

\begin{theorem}
Let $f$ be a functions which satisfy the assumptions of Lemma 1. If $%
\left\vert f^{\prime \prime }\right\vert $ belongs to the class $Q\left(
I\right) ,$ then the following inequality holds:%
\begin{eqnarray*}
&&\left\vert \frac{1}{b-a)}\int_{a}^{b}f(x)dx-\frac{1}{6}\left[
f(a)+4f\left( \frac{a+b}{2}\right) +f(b)\right] \right\vert \\
&\leq &\frac{(b-a)^{2}}{6}\left( \frac{12\ln 2-8\ln 3+1}{2}\right) \left[
\left\vert f^{\prime \prime }\left( a\right) \right\vert +\left\vert
f^{\prime \prime }\left( b\right) \right\vert \right] .
\end{eqnarray*}
\end{theorem}

\begin{proof}
From Lemma 1 and using the properties of modulus, we have%
\begin{eqnarray*}
&&\left\vert \frac{1}{b-a)}\int_{a}^{b}f(x)dx-\frac{1}{6}\left[
f(a)+4f\left( \frac{a+b}{2}\right) +f(b)\right] \right\vert \\
&\leq &(b-a)^{2}\left\{ \int_{0}^{\frac{1}{2}}\left\vert \frac{1}{6}t\left(
3t-1\right) \right\vert \left\vert f^{\prime \prime }(tb+\left( 1-t\right)
a)\right\vert dt\right. \\
&&\left. +\int_{\frac{1}{2}}^{1}\left\vert \frac{1}{6}\left( t-1\right)
\left( 3t-2\right) \right\vert \left\vert f^{\prime \prime }(tb+\left(
1-t\right) a)\right\vert dt\right\} \\
&=&\frac{(b-a)^{2}}{6}\left\{ \int_{0}^{\frac{1}{3}}t\left( 1-3t\right)
\left\vert f^{\prime \prime }(tb+\left( 1-t\right) a)\right\vert dt+\int_{%
\frac{1}{3}}^{\frac{1}{2}}t\left( 3t-1\right) \left\vert f^{\prime \prime
}(tb+\left( 1-t\right) a)\right\vert dt\right. \\
&&\left. +\int_{\frac{1}{2}}^{\frac{2}{3}}\left( 1-t\right) \left(
2-3t\right) \left\vert f^{\prime \prime }(tb+\left( 1-t\right) a)\right\vert
dt+\int_{\frac{2}{3}}^{1}\left( 1-t\right) \left( 3t-2\right) \left\vert
f^{\prime \prime }(tb+\left( 1-t\right) a)\right\vert dt\right\} .
\end{eqnarray*}%
Since $\left\vert f^{\prime \prime }\right\vert $ belongs to the class $%
Q\left( I\right) ,$ we can write%
\begin{eqnarray*}
&&\left\vert \frac{1}{b-a)}\int_{a}^{b}f(x)dx-\frac{1}{6}\left[
f(a)+4f\left( \frac{a+b}{2}\right) +f(b)\right] \right\vert \\
&\leq &\frac{(b-a)^{2}}{6}\left\{ \int_{0}^{\frac{1}{3}}t\left( 1-3t\right)
\left( \frac{\left\vert f^{\prime \prime }\left( b\right) \right\vert }{t}+%
\frac{\left\vert f^{\prime \prime }\left( a\right) \right\vert }{1-t}\right)
dt+\int_{\frac{1}{3}}^{\frac{1}{2}}t\left( 3t-1\right) \left( \frac{%
\left\vert f^{\prime \prime }\left( b\right) \right\vert }{t}+\frac{%
\left\vert f^{\prime \prime }\left( a\right) \right\vert }{1-t}\right)
dt\right. \\
&&\left. +\int_{\frac{1}{2}}^{\frac{2}{3}}\left( 1-t\right) \left(
2-3t\right) \left( \frac{\left\vert f^{\prime \prime }\left( b\right)
\right\vert }{t}+\frac{\left\vert f^{\prime \prime }\left( a\right)
\right\vert }{1-t}\right) dt+\int_{\frac{2}{3}}^{1}\left( 1-t\right) \left(
3t-2\right) \left( \frac{\left\vert f^{\prime \prime }\left( b\right)
\right\vert }{t}+\frac{\left\vert f^{\prime \prime }\left( a\right)
\right\vert }{1-t}\right) dt\right\}
\end{eqnarray*}%
Computing the above integrals, we get te desired result.
\end{proof}

\begin{corollary}
In Theorem 1, if we take $f\left( a\right) =f\left( b\right) =f\left( \frac{%
a+b}{2}\right) ,$ then we have%
\begin{equation*}
\left\vert \frac{1}{b-a)}\int_{a}^{b}f(x)dx-f\left( \frac{a+b}{2}\right)
\right\vert \leq \frac{(b-a)^{2}}{6}\left( \frac{12\ln 2-8\ln 3+1}{2}\right) %
\left[ \left\vert f^{\prime \prime }\left( a\right) \right\vert +\left\vert
f^{\prime \prime }\left( b\right) \right\vert \right] .
\end{equation*}%
For $M>0,$ if $\left\vert f^{\prime \prime }\left( x\right) \right\vert <M,$
for all $x\in \left[ a,b\right] ,$ then we have%
\begin{equation*}
\left\vert \frac{1}{b-a)}\int_{a}^{b}f(x)dx-f\left( \frac{a+b}{2}\right)
\right\vert \leq \frac{(b-a)^{2}}{3}\left( \frac{12\ln 2-8\ln 3+1}{2}\right)
M.
\end{equation*}
\end{corollary}

\begin{theorem}
Let $f$ be a functions which satisfy the assumptions of Lemma 1. If $%
\left\vert f^{\prime \prime }\right\vert ^{q}$ belongs to the class $Q\left(
I\right) ,$ then the following inequality holds:%
\begin{eqnarray*}
&&\left\vert \frac{1}{b-a)}\int_{a}^{b}f(x)dx-\frac{1}{6}\left[
f(a)+4f\left( \frac{a+b}{2}\right) +f(b)\right] \right\vert \\
&\leq &\frac{(b-a)^{2}}{6}\left( \frac{1}{27}\right) ^{1-\frac{1}{q}}\left\{ %
\left[ \left( 6\ln 2-4\ln 3+\frac{7}{24}\right) \left\vert f^{\prime \prime
}\left( a\right) \right\vert ^{q}+\frac{5\left\vert f^{\prime \prime }\left(
b\right) \right\vert ^{q}}{24}\right] ^{\frac{1}{q}}\right. \\
&&\left. +\left[ \left( 6\ln 2-4\ln 3+\frac{7}{24}\right) \left\vert
f^{\prime \prime }\left( b\right) \right\vert ^{q}+\frac{5\left\vert
f^{\prime \prime }\left( a\right) \right\vert ^{q}}{24}\right] ^{\frac{1}{q}%
}\right\} .
\end{eqnarray*}%
where $q\geq 1.$
\end{theorem}

\begin{proof}
From Lemma 1 and using the Power-mean inequality, we have 
\begin{eqnarray*}
&&\left\vert \frac{1}{b-a)}\int_{a}^{b}f(x)dx-\frac{1}{6}\left[
f(a)+4f\left( \frac{a+b}{2}\right) +f(b)\right] \right\vert \\
&\leq &\frac{(b-a)^{2}}{6}\left\{ \left( \int_{0}^{\frac{1}{2}}\left\vert
t\left( 3t-1\right) \right\vert dt\right) ^{1-\frac{1}{q}}\left( \int_{0}^{%
\frac{1}{2}}\left\vert t\left( 3t-1\right) \right\vert \left\vert f^{\prime
\prime }(tb+\left( 1-t\right) a)\right\vert ^{q}dt\right) ^{\frac{1}{q}%
}\right. \\
&&+\left. \left( \int_{\frac{1}{2}}^{1}\left\vert \left( t-1\right) \left(
3t-2\right) \right\vert dt\right) ^{1-\frac{1}{q}}\left( \int_{\frac{1}{2}%
}^{1}\left( t-1\right) \left( 3t-2\right) \left\vert f^{\prime \prime
}(tb+\left( 1-t\right) a)\right\vert ^{q}dt\right) ^{\frac{1}{q}}\right\}
\end{eqnarray*}%
Since $\left\vert f^{\prime \prime }\right\vert $ belongs to the class $%
Q\left( I\right) $ and by computing the above integrals, we deduce 
\begin{eqnarray*}
&&\left\vert \frac{1}{b-a)}\int_{a}^{b}f(x)dx-\frac{1}{6}\left[
f(a)+4f\left( \frac{a+b}{2}\right) +f(b)\right] \right\vert \\
&\leq &\frac{(b-a)^{2}}{6}\left( \frac{1}{27}\right) ^{1-\frac{1}{q}}\left\{
\left( \left( 6\ln 2-4\ln 3+\frac{7}{24}\right) \left\vert f^{\prime \prime
}\left( a\right) \right\vert ^{q}+\frac{5\left\vert f^{\prime \prime }\left(
b\right) \right\vert ^{q}}{24}\right) ^{\frac{1}{q}}\right. \\
&&\left. +\left( \left( 6\ln 2-4\ln 3+\frac{7}{24}\right) \left\vert
f^{\prime \prime }\left( b\right) \right\vert ^{q}+\frac{5\left\vert
f^{\prime \prime }\left( a\right) \right\vert ^{q}}{24}\right) ^{\frac{1}{q}%
}\right\}
\end{eqnarray*}%
which is the desired result.
\end{proof}

\section{Applications to Numerical Integration}

Let $d$ be a division of the interval $\left[ a,b\right] ,$ i.e., $%
d:a=x_{0}<x_{1}<...<x_{n-1}<x_{n}=b,$ $h_{i}=\frac{\left(
x_{i+1}-x_{i}\right) }{2}$ and consider the Simpson's formulae%
\begin{equation*}
S\left( f,d\right) =\dsum\limits_{i=0}^{n-1}\frac{f\left( x_{i}\right)
+f\left( x_{i}+h_{i}\right) +f\left( x_{i+1}\right) }{6}\left(
x_{i+1}-x_{i}\right) .
\end{equation*}%
It is well-known that if the mapping $f:\left[ a,b\right] \rightarrow 
\mathbb{R}
,$ is differentiable such that $f^{\left( 4\right) }\left( x\right) $ exists
on $\left( a,b\right) $ and $M=\max\limits_{x\in \left( a,b\right)
}\left\vert f^{\left( 4\right) }\left( x\right) \right\vert <\infty $, then 
\begin{equation*}
I=\dint\limits_{a}^{b}f\left( x\right) dx=S\left( f,d\right) +E_{s}\left(
f,d\right) ,
\end{equation*}%
where the approximation error $E_{s}\left( f,d\right) $ of the integral $I$
by the Simpson's formulae $S\left( f,d\right) $ satisfies 
\begin{equation*}
E_{s}\left( f,d\right) \leq \frac{M}{90}\dsum\limits_{i=0}^{n-1}\left(
x_{i+1}-x_{i}\right) ^{5}.
\end{equation*}%
Now we will give estimation for remainder term $E\left( f,d\right) $ in
terms of the second derivative.

\begin{proposition}
Under the conditions of Theorem 1, then for every division $d$ of $\left[ a,b%
\right] ,$ the following holds:%
\begin{equation*}
\left\vert E_{s}\left( f,d\right) \right\vert \leq \left( \frac{12\ln 2-8\ln
3+1}{12}\right) \dsum\limits_{i=0}^{n-1}\left( x_{i+1}-x_{i}\right) ^{2}%
\left[ \left\vert f^{\prime \prime }\left( x_{i}\right) \right\vert
+\left\vert f^{\prime \prime }\left( x_{i+1}\right) \right\vert \right] .
\end{equation*}
\end{proposition}

\begin{proof}
Applying Theorem 1 on the subintervals $\left[ x_{i},x_{i+1}\right] ,$ $%
\left( i=0,1,...,n-1\right) $ of the division $d$ and summing over $i$ from $%
0$ to $n-1$, we get the required result.
\end{proof}

\end{document}